\documentclass{amsart}
\usepackage{amssymb}
\title[Jacobi Sums modulo  Prime Powers]{Evaluating Prime Power Gauss and Jacobi Sums}
\author{Misty Long}
\address{ Department of Mathematics\\
          Kansas State University\\
          Manhattan, KS 66506}
\email{milong@math.ksu.edu}
\author{Vincent Pigno}
\address{ Department of Mathematics\\
         Kansas State University\\
         Manhattan, KS 66506}
\email{pignov@math.ksu.edu}

\author{Christopher Pinner}
\address{ Department of Mathematics\\
         Kansas State University\\
         Manhattan, KS 66506}
\email{pinner@math.ksu.edu}

\date{\today}

\newcommand{\be}{\begin{equation}}
\newcommand{\ee}{\end{equation}}
\newcommand{\ba}{\begin{align}}
\newcommand{\ea}{\end{align}}

\begin{document}
\begin{abstract} We show that for any mod $p^m$ characters, $\chi_1, \dots, \chi_k,$  the  Jacobi  sum,
$$
\sum_{x_1=1}^{p^m}\dots \sum_{\substack{x_k=1\\x_1+\dots+x_k=B}}^{p^m}\chi_1(x_1)\dots \chi_k(x_k),
$$
has a simple evaluation when $m$ is sufficiently large (for $m\geq 2$ if $p\nmid B$). As part of the proof we give a simple evaluation of the mod $p^m$ Gauss sums when $m\geq 2$.
\end{abstract}

\maketitle
\newtheorem{theorem}{Theorem}[section]
\newtheorem{corollary}{Corollary}[section]
\newtheorem{lemma}{Lemma}[section]
\newtheorem{conjecture}{Conjecture}[section]

\section{Introduction}
For multiplicative characters $\chi_1$ and $\chi_2$ mod $q$ one defines the classical Jacobi sum by
\be \label{jacobisum}
J(\chi_1,\chi_2,q):= \sum_{x=1}^{q} \chi_1(x) \chi_2(1-x).
\ee
More generally for $k$ characters  $\chi_1, \dots, \chi_k$ mod $q$ one can define
\be\label{genjacobi1}
J(\chi_1,\dots, \chi_k,q)=\sum_{x_1=1}^{q}\dots \sum_{\substack{x_k=1\\x_1+\dots+x_k=1}}^q\chi_1(x_1)\cdots \chi_k(x_k).
\ee
 If the $\chi_i$ are mod $rs$ characters with $(r,s)=1$ then,  writing $\chi_i=\chi_i' \chi_i''$  where $\chi_i'$ and $\chi_i''$ are mod $r$ and mod $s$  characters respectively,  it is readily seen (e.g. \cite[Lemma 2]{Wenpeng2}) that  
$$ J(\chi_1,\ldots,\chi_k,rs)=J(\chi_1',\ldots ,\chi_k',r) J(\chi_1'',\ldots ,\chi_k'',s). $$
Hence, one usually only considers the case of  prime power moduli $q=p^m$.

 Zhang \&  Yao \cite{Wenpeng1} showed that the sums \eqref{jacobisum} can in fact  be evaluated explicitly  when $m$ is even  (and $\chi_1$, $\chi_2$ and $\chi_1 \chi_2$ are primitive mod $p^m$). Working with a slightly  more general binomial character  sum the authors \cite{PPP}  showed that techniques of Cochrane \& Zheng \cite{cz} can be used to obtain an evaluation of 
\eqref{jacobisum}  for any  $m>1$  ($p$ an odd prime). Zhang and Xu \cite{Wenpeng2} considered the general case, \eqref{genjacobi1},   obtaining   (assuming that   $\chi, \chi^{n_1}, \dots, \chi^{n_k},$ and $ \chi^{n_1+ \dots + n_{k}}$ are primitive characters modulo $p^m$) 
\be \label{genjacobi2}
J(\chi^{n_1},\ldots ,\chi^{n_k},p^m) =p^{\frac{1}{2}(k-1)m}\overline{\chi}(u^u)\chi(n_1^{n_1}\dots n_k^{n_k}),\;\; u:=n_1+\cdots +n_k, 
\ee
 when $m$ is even, and when the  $m, k, n_1, \dots , n_k$ are all odd
\be \label{genjacobi3}
J(\chi^{n_1},\ldots ,\chi^{n_k},p^m)  =p^{\frac{1}{2}(k-1)m} \overline{\chi}(u^u)\chi(n_1^{n_1}\dots n_{k-1}^{n_{k-1}}) \begin{cases} \varepsilon_p^{k-1} \left(\frac{un_1 \dots n_k}{p} \right), & \text{if $p\neq 2$;} \\ \left( \frac{2}{u n_1 \dots n_k}\right) & \text{if } p=2, \end{cases}
\ee
where  $\left(\frac{m}{n}\right)$ is the Jacobi symbol and  (defined more generally for later use)
 \be \label{quadep}\varepsilon_{p^m} := \begin{cases} 1, &\text{if } p^m\equiv 1\,\text{mod }4,  \\i, & \text{if }p^m\equiv 3\,\text{mod }4.
\end{cases}
\ee
 In this paper we give an evaluation for all $m>1$ (i.e. irrespective of the parity of $k$ and the $n_i$). In fact we evaluate the slightly more general sum 
$$
J_B(\chi_1,\dots, \chi_k,p^m)=\sum_{x_1=1}^{p^m}\dots \sum_{\substack{x_k=1\\x_1+\dots+x_k=B}}^{p^m}\chi_1(x_1)\cdots \chi_k(x_k).
$$
Of course when $B=p^nB'$, $p\nmid B'$  the simple change of variables $x_i \mapsto B'x_i$ gives $$J_B(\chi_1,\dots, \chi_k,p^m)=\chi_1 \cdots \chi_k(B') J_{p^n}(\chi_1,\dots, \chi_k,p^m).$$ For example $
J_B(\chi_1,\dots, \chi_k,p^m)=\chi_1 \cdots \chi_k(B) J(\chi_1,\dots, \chi_k,p^m)
$ when  $p \nmid B$. From the change of variables $x_i \mapsto -x_k x_i$, $1\leq i <k$ one also
sees that
$$ J_{p^m} (\chi_1,\ldots ,\chi_k,p^m) =\begin{cases} \phi(p^m)\chi_k(-1)  J(\chi_1,\ldots ,\chi_{k-1},p^m), & \text{ if $\chi_1\cdots \chi_k=\chi_0$, } \\
 0, & \text{  if $\chi_1\cdots \chi_k\neq \chi_0$, } \end{cases} $$
where $\chi_0$  denotes  the principal character, so  we assume that $B=p^n$ with $n<m$.

\begin{theorem}\label{maintheorem}
Let $p$ be a  prime and $m\geq n+ 2$.
Suppose that $\chi_1, \dots, \chi_k$ are $k\geq 2$ characters mod $p^m$ with at least one of them primitive. 

If the  $\chi_1,\ldots, \chi_k$  are not all  primitive mod $p^m$ or
$\chi_1\dots \chi_k$  is not induced by a primitive mod  $p^{m-n}$  character,
then $J(\chi_1,\dots, \chi_k,p^m) =0$.

 If  the $\chi_1,\ldots, \chi_k$ are primitive  mod $p^m$ and $\chi_1\cdots \chi_k$ is primitive mod $p^{m-n}$,  then
\be \label{evaluation}
J_{p^n}(\chi_1,\dots, \chi_k,p^m) =  p^{\frac{1}{2}(m(k-1)+n)}\frac{\chi_1(c_1)\cdots \chi_k(c_k)}{\chi_1\cdots \chi_k(v)}\, \delta, \ee
where  for $p$ odd
$$ \delta = \left(\frac{-2r}{p}\right)^{m(k-1)+n}\left(\frac{v}{p}\right)^{m-n} \left(\frac{c_1\cdots c_k}{p}\right)^m \varepsilon_{p^m}^k \varepsilon_{p^{m-n}}^{-1}, $$
and for $p=2$  and  $m-n\geq 5$, 
\be \label{p=2case} \delta = \left(\frac{2}{v}\right)^{m-n} \left(\frac{2}{c_1\cdots c_k}\right)^m 
\omega^{(2^n-1)v}, \ee
with  $\varepsilon_{p^m}$ as defined in \eqref{quadep}, the
$r$ and  $c_i$   as in \eqref{conditions} and \eqref{conditions2} or \eqref{conditions2'}  below, and
\be \label{defw}   v:=p^{-n}(c_1+\dots +c_k),\;\;\;  \omega := e^{\pi i/4}. \ee
 For $m\geq 5$ and $m-n=2,3$ or $4$ the formula \eqref{p=2case}  for $\delta$ 
should be multiplied by $\omega$, $\omega^{1+\chi_1\cdots \chi_k (-1)}$ or $\chi_1\cdots \chi_k (-1)\omega^{2v}$
respectively.

\end{theorem}
Of course it is  natural  to assume that at least one of the $\chi_1,\ldots ,\chi_k$ is primitive, 
otherwise we can reduce the sum to a mod $p^{m-1}$ sum.
For $n=0$  and $\chi_1,\ldots ,\chi_k$ and $\chi_1 \cdots \chi_k$ all primitive mod $p^m$ our result simplifies to  
$$ J(\chi_1,\dots, \chi_k,p^m) =  p^{\frac{m(k-1)}{2}}\frac{\chi_1(c_1)\cdots \chi_k(c_k)}{\chi_1\cdots \chi_k(v)}\; \delta,\;\;\; v=c_1+\cdots  +c_k, $$
with
$$ \delta = \begin{cases}  1, &  \text{if $m$ is even,} \\
\left(\frac{vc _1\cdots c_k}{p}\right) \left(\frac{-2rc}{p}\right)^{k-1}\varepsilon_p^{k-1}, & \text{if $m$ is odd and $p\neq 2$, } \\
\left(\frac{2}{vc_1\cdots c_k}\right), & \text{if $m\geq 5$ is odd and $p=2$.} \end{cases}$$
In the remaining $n=0$  case,  $p=2$, $m=3$ we 
have $J(\chi_1,\ldots ,\chi_k,2^3)=2^{\frac{3}{2}(k-1)}(-1)^{\lfloor \frac{\ell}{2}\rfloor}$ where $\ell$ denotes  the number of characters $1\leq i\leq k$  with $\chi_i(-1)=-1$.

 When the $\chi_i=\chi^{n_i}$ for
some primitive mod $p^m$ character $\chi$ we can write $c_i=n_ic$  (where $c$ is determined by 
 $\chi(a)$ as in \eqref{conditions2} or \eqref{conditions2'}) and we recover the form
\eqref{genjacobi2} and \eqref{genjacobi3} with the addition of a factor $\left(\frac{-2rc}{p}\right)^{k-1}$ for $p\neq 2$, $m$ odd, which of course can be ignored when $k$ is odd as assumed in \cite{Wenpeng2}.

For completeness  we observe that in the few remaining $m\geq n+2$  cases \eqref{evaluation} becomes
$$J_{p^n}(\chi_1,\ldots,\chi_k,p^m)=2^{\frac{1}{2}(m(k-1)+n)}\begin{cases}
-i\omega^{k-\sum_{i=1}^k \chi_i(-1)}, & \text{ if $m=3$, $n=1$,} \\ 
\omega^{\chi_1\cdots \chi_k(-1)-1-v} \prod_{i=1}^k\chi_i(-c_i), & \text{ if $m=4$, $n=1$,} \\ 
i^{1-v} \prod_{i=1}^k\chi_i(c_i), & \text{ if $m=4$, $n=2$.}\end{cases} $$

Our  proof  of Theorem \ref{maintheorem}  involves expressing the Jacobi sum \eqref{genjacobi1} in terms of  classical Gauss sums
 \be \label{defGauss} G(\chi,p^m):=\sum_{x=1}^{p^m} \chi(x)e_{p^m}(x), \ee
where $\chi$ is a mod $p^m$ character and $e_{y}(x):=e^{2 \pi i x / y}$. 
%For $m>1$  these
%Gauss sums can be evaluated explicitly using the method of Cochrane and Zheng \cite{cz}.
Writing \eqref{jacobisum} in terms of Gauss sums is well known for the mod $p$ sums
and the corresponding result for  \eqref{genjacobi1}  can be found, along with
many other  properties of Jacobi sums, in Berndt, R. J. Evans and K. S. Williams  \cite[Theorem 2.1.3 \& Theorem 10.3.1 ] {BerndtBk}  or Lidl-Niederreiter \cite[Theorem 5.21]{LidlNiedBk}. There the results are
stated for sums over finite fields, $\mathbb F_{p^m}$,  so it is not surprising that such
expressions exist in the less studied mod $p^m$ case.  When $\chi_1,\ldots ,\chi_k$ and $\chi_1\cdots \chi_k$ are primitive,  Zhang \& Yao  \cite[Lemma 3]{Wenpeng1} for $k=2$, and  Zhang and Xu \cite[Lemma 1]{Wenpeng2}  for
general $k$,  showed that
\be \label{symmetricform}
J(\chi_1,\dots, \chi_k,p^m)= \frac{ \prod_{i=1}^{k} G(\chi_i , p^m)}{G(\chi_1 \dots \chi_k,p^{m})}.
\ee 
In Theorem \ref{intermsofgauss}    we obtain a similar expansion for  $J_{p^n}(\chi_1,\ldots ,\chi_k,p^m)$.
As we show in Theorem \ref{gausssumthm}  the mod $p^m$ 
Gauss sums can be evaluated explicitly using the method of Cochrane and Zheng \cite{cz} when $m\geq 2$.

For $m=n+1$  (with at  least one $\chi_i$ primitive) the Jacobi sum is  still zero   unless all the $\chi_i$ are primitive mod $p^m$ and $\chi_1\cdots \chi_k$ is a mod $p$
character. Then  we can say that $|J_{p^n}(\chi_1,\ldots ,\chi_k,p^m)|=p^{\frac{1}{2}mk-1}$ if $\chi_1\cdots \chi_k=\chi_0$ and $p^{\frac{1}{2}(mk-1)}$ otherwise, but an explicit
evaluation in the latter case  is equivalent to  an explicit evaluation of the mod $p$ Gauss sum $G(\chi_1\cdots \chi_k,p)$ when $m\geq 2$.

\section{Gauss Sums}\label{GaussSums}

In order to use the result from \cite{cz2} we must first define some terms. For $p$ odd let $a$ be a primitive root mod $p^m$.  We define the integers  $r$, and $R_j$ by 
\be \label{conditions}
a^{\phi(p)}=1+rp, \hskip 3ex a^{\phi \left(p^{j}\right)}=1+R_j p^{j}. 
\ee
Note, $p\nmid r$ and for $j\geq i$, \be \label{poddcong} R_j\equiv R_i\text{  mod }p^i. \ee For a  character $\chi_i$ mod $p^m$ we define $c_i$ by 
\be\label{conditions2} \chi_i(a)=e_{\phi(p^m)}(c_i), \ee with $1\leq c_i \leq \phi(p^m)$.
Note,  $p\nmid c_i$ exactly when $\chi_i$ is primitive. For $p=2$  and $m\geq 3$ we need two
generators $-1$ and $a=5$ for $\mathbb Z_{2^m}^*$ and define $R_j$, $j\geq 2$,  and $c_i$ by
\be \label{conditions2'}  a^{2^{j-2}}=1+R_j2^j,   \hspace{3ex}  \chi_i(a) = e_{2^{m-2}}(c_i), \ee
with $\chi_i$ primitive exactly  when $2\nmid c_i$. Noting that $R_i^2 \equiv 1$ mod $8$, we 
get 
\be \label{p2cong}
R_{i+1} =R_{i}+ 2^{i-1}R_{i}^2  \equiv R_{i} + 2^{i-1}\text{ mod }2^{i+2}.
\ee
For $ j\geq i+2 $  this  gives the relationships,
\be \label{p2cong'}
 R_j  \equiv R_{i+2} \equiv R_{i+1}+2^i \equiv (R_i+2^{i-1})+2^{i} \equiv R_i-2^{i-1} \text{ mod } 2^{i+1} 
\ee
and
\be \label{cong3}
 R_j  \equiv (R_{i-1} +2^{i-2}) -2^{i-1} \equiv R_{i-1}-2^{i-2} \text{ mod } 2^{i+1}.
\ee
We shall  need an explicit evaluation of the mod $p^m$, $m\geq 2$,   Gauss sums.  The form we use
comes from applying the technique of Cochrane \& Zheng \cite{cz} as formulated in \cite{PP}.
 For odd $p$  this is  essentially the same  as  \cite[\S 9]{cz2} but for $p=2$ seems new. Variations can be found in Odoni \cite{Odoni} and Mauclaire \cite{Mau1}  (see also \cite[Chapter 1]{BerndtBk}).

\begin{theorem} \label{gausssumthm} Suppose that  $\chi$ is a  mod $p^m$ character with  $m\geq 2$. If $\chi$ is imprimitive, then $G(\chi,p^m)=0$.
If $\chi $ is primitive,  then 
\be \label{GaussSumEval}
G(\chi,p^m) =p^{\frac{m}{2}}\chi\left(-cR_{j}^{-1}\right)e_{p^m}\left( -cR_{j}^{-1} \right)\begin{cases} \left(\frac{-2rc}{p}\right)^m \varepsilon_{p^m}, & \text{if $p\neq 2$,}\\  
\left( \frac{2}{c} \right)^m \omega^c, & \text{if $p= 2$ and $m\geq 5$,}
  \end{cases}
\ee
for any $j\geq\lceil \frac{m}{2}\rceil$ when $p$ is odd and any $j\geq\lceil \frac{m}{2}\rceil+2$
 when  $p=2$.

For the remaining cases 
\be \label{special} G(\chi,2^m)=2^{\frac{m}{2}}\begin{cases}  i, & \text{ if $m=2$,} \\
\omega^{1-\chi(-1)}, & \text{ if $m=3$,} \\
\chi(-c)e_{16}(-c), & \text{ if $m=4$.} \end{cases} 
\ee
Here $x^{-1}$ denotes the inverse of $x$ mod $p^m$,  and $r$, $R_j$ and $c$ are as in \eqref{conditions} and \eqref{conditions2} or \eqref{conditions2'} and $\omega$ as in \eqref{defw}.

\end{theorem}

\begin{proof}   When $p$ is odd  \cite[Theorem 2.1]{PP} gives
 $$ G(\chi,p^m) = p^{m/2} \chi(\alpha) e_{p^m}(\alpha) \left(\frac{-2rc}{p^m}\right) \varepsilon_{p^m}$$
where $\alpha$ is a solution of
\be \label{gausschar} c+R_{J}x \equiv 0 \text{ mod } p^{J},\;\;\; J:=\left\lceil \frac{m}{2}\right\rceil, \ee
(and zero if no solution exists). If $p\mid c$ there is no solution and $G(\chi,p^m)=0$. If $p\nmid c$ 
by \eqref{poddcong} we may take $\alpha=-cR_J^{-1} \equiv -cR_j^{-1}$ mod $p^J$ for any $j\geq J$.
If  $p=2$, $m\geq 6$, and $\chi$ is primitive,  then \cite[Theorem 5.1]{PP}  gives
$$ G(\chi,p^m) = 2^{m/2} \chi (\alpha) e_{2^m}(\alpha) \begin{cases} 1, & \text{ if $m$ is even,} \\ \left(\frac{1+ (-1)^{\lambda} i^{R_Jc}}{\sqrt{2}}\right), & \text{ if $m$ is odd, }\end{cases} $$
where $\alpha$ is a solution to 
\be \label{gausschar2} c+R_{J}x \equiv 0 \text{ mod } 2^{\lfloor \frac{m}{2}\rfloor}, \ee
and  $ c+R_{J}\alpha  = 2^{\lfloor \frac{m}{2}\rfloor}\lambda $
(and zero if there is  no solution or $\chi$ is  imprimitive). If $2\nmid c$  and $j\geq J+2$ 
then (using \eqref{p2cong'}  and $R_j\equiv -1$ mod 4)  we can take
$$\alpha \equiv   -cR_J^{-1}\equiv -c(R_j+2^{J-1})^{-1} \equiv -c(R_j^{-1}-2^{J-1}) \text{ mod } 2^{J+1},  $$ 
and
$$ \chi(\alpha)e_{2^m}(\alpha) = \chi(-cR_j^{-1})e_{2^m}(-cR_j^{-1})  \chi(1-R_j2^{J-1})e_{2^m} (c2^{J-1}), $$ 
where, checking the four possible $c$ mod 8,
$$ \left(\frac{1+ (-1)^{\lambda} i^{R_Jc}}{\sqrt{2}}\right)=\left(\frac{1- i^{c}}{\sqrt{2}}\right)= \omega^{-c}  \left(\frac{2}{c}\right). $$
 Now
$$ e_{2^m} (c2^{J-1})= e_{2^{m-2}} (c2^{J-3})=\chi\left(5^{2^{J-3}}\right)=\chi\left(1+R_{J-1}2^{J-1}\right), $$
where, since $R_j\equiv R_{J-1}-2^{J-2}$ mod $2^{J+1}$,
\begin{align*}  \left(1-R_j2^{J-1}\right)\left(1+R_{J-1}2^{J-1}\right) & = 1+(R_{J-1}-R_j)2^{J-1} -R_jR_{J-1} 2^{2J-2} \\  & \equiv 1+2^{2J-3} +R_{J-1}2^{2J-2} \equiv 1+R_{2J-3}2^{2J-3} \text{ mod } 2^m. \end{align*}
Hence 
$$  \chi(1-R_j2^{J-1})e_{2^m} (c2^{J-1})= \chi\left(5^{2^{2J-5}}\right)=e_{2^{m-2}}(c2^{2J-5}) =\begin{cases} \omega^c, & \text{ if $m$ is even,} \\
 \omega^{2c}, & \text{ if $m$ is odd.} \end{cases} $$
One can check numerically that the formula still holds for the $2^{m-2}$ primitive mod $2^m$  characters when $m=5$. For $m=2,3,4$  one has \eqref{special}   instead of $2i\omega$,
$2^{\frac{3}{2}}\omega^2$, $2^2\chi (c) e_{2^4}(c)\omega^c$ (so our formula \eqref{GaussSumEval}
requires an extra factor $\omega^{-1}$, $\omega^{-1-\chi(-1)}$ or $\chi(-1)\omega^{-2c}$
respectively).

\end{proof}

We shall need the counterpart of \eqref{symmetricform} for the $J_{p^n}(\chi_1,\ldots ,\chi_k)$. We state  a less symmetrical version  to allow weaker  assumptions
on the $\chi_i$:
\begin{theorem} \label{intermsofgauss} Suppose that $\chi_1, \dots, \chi_k$ are characters mod $p^m$ with $m>n$ and  $\chi_k$ primitive mod $p^m$.
If $\chi_1 \cdots \chi_k$ is a  mod $p^{m-n}$ character,  then
\be
J_{p^n}(\chi_1,\ldots,\chi_k,p^m) =p^n \, \frac{\overline{G(\chi_1 \cdots \chi_k, p^{m-n})}}{\overline{G(\chi_k, p^m)}}   \prod_{i=1}^{k-1} G(\chi_i, p^m).
\ee
If   $\chi_1 \cdots \chi_k$ is not a  mod $p^{m-n}$ character,  then $J_{p^n}(\chi_1,\ldots,\chi_k,p^m)=0$.
\end{theorem}
From well known properties of Gauss sums  (see for example Section $1.6$ of \cite{BerndtBk}),
\be \label{gaussmod}
\mid G(\chi,p^j) \mid=\begin{cases} p^{j/2}, & \text{if $\chi$ is primitive mod $p^j$,}\\
1, & \text{if $\chi=\chi_0$ and $j=1$,}\\
0, & \text{otherwise,}
\end{cases}
\ee
when $\chi_1 \cdots \chi_k$ is a  primitive mod $p^{m-n}$ character and at least one of the $\chi_i$ is a primitive mod $p^m$ character we immediately  obtain the symmetric form
\be \label{symmetricform2}
J_{p^n}(\chi_1,\dots, \chi_k,p^m)=\frac{ \prod_{i=1}^{k} G(\chi_i , p^m)}{G(\chi_1 \dots \chi_k,p^{m-n})}.
\ee
In particular we recover \eqref{symmetricform} under the sole assumption that $\chi_1 \cdots \chi_k$ is a primitive mod $p^m$ character.

\begin{proof}
We first note that if $\chi$ is a primitive character mod $p^j$, $j \geq 1$, then
$$
\sum_{y=1}^{p^j} \chi(y) e_{p^j}(Ay)=\overline{\chi}(A) G(\chi, p^j).
$$
 Indeed, for $p\nmid A$ this is plain
from $y\mapsto A^{-1}y$. If $p\mid A$ and $j=1$ the sum equals $\sum_{y=1}^{p}\chi(y)=0$. For $j\geq 2$ as $\chi$ is primitive there exists a $ z \equiv 1$ mod $p^{j-1}$ with  $\chi(z)\neq 1$, (there must be some $a\equiv b$ mod $p^{j-1}$ with  $\chi(a)\neq\chi(b)$, and we can take  $z=ab^{-1}$) so
\be \label{sumiszero}
 \sum_{y=1}^{p^j} \chi(y)e_{p^j}(Ay)=\sum_{y=1}^{p^j} \chi(zy)e_{p^j}(Azy)=\chi(z)\sum_{y=1}^{p^j} \chi(y)e_{p^j}(Ay)
\ee
and $\sum_{y=1}^{p^j} \chi(y)e_{p^j}(Ay)=0$.

Hence if $\chi_k$ is a primitive character mod $p^m$ we have
\begin{align*}
& \overline{\chi}_k(-1) G(\overline{\chi}_k, p^m)  \sum_{x_1=1}^{p^m} \dots \sum_{x_{k-1}=1}^{p^m} \chi_1(x_1) \dots \chi_{k-1}(x_{k-1}) \chi_k(p^n- x_1- \dots - x_{k-1})\\
& = \overline{\chi}_k(-1) \sum_{x_1=1}^{p^m} \dots \sum_{x_{k-1}=1}^{p^m} \chi_1(x_1) \dots \chi_{k-1}(x_{k-1}) \sum_{y=1}^{p^m} \overline{\chi}_k(y) e_{p^m}((p^n- x_1- \dots - x_{k-1})y)\\
& = \sum_{\substack{y=1\\  p \nmid y}}^{p^m} \overline{\chi}_k(-y) e_{p^m}(p^ny) \left( \sum_{x_1=1}^{p^m} \chi_1(x_1) e_{p^m}(-x_1 y) \dots   \sum_{x_{k-1}=1}^{p^m} \chi_{k-1}(x_{k-1}) e_{p^m}(-x_{k-1} y) \right) \\
& = \sum_{\substack{y=1 \\ p\nmid y}}^{p^m} \overline{\chi_1 \dots \chi}_k (-y) e_{p^m}(p^ny)
\left( \sum_{x_1=1}^{p^m} \chi_1(x_1) e_{p^m}(x_1) \dots   \sum_{x_{k-1}=1}^{p^m} \chi_{k-1}(x_{k-1}) e_{p^m}(x_{k-1}) \right)\\
& = \overline{\chi_1 \dots \chi}_{k} (-1)\sum_{\substack{y=1 \\ p\nmid y}}^{p^m} \overline{\chi_1 \dots \chi}_k (y) e_{p^m}(p^ny) \prod_{i=1}^{k-1} G(\chi_i, p^m).
\end{align*}
If $m>n$ and $\overline{\chi_1 \dots \chi}_k$ is a mod $p^{m-n}$ character,  then
$$
\sum_{\substack{y=1 \\ p\nmid y}}^{p^m} \overline{\chi_1 \dots \chi}_k (y) e_{p^m}(p^ny)=p^n\sum_{\substack{y=1 \\ p\nmid y}}^{p^{m-n}} \overline{\chi_1 \dots \chi}_k (y) e_{p^{m-n}}(y) =p^n G(\overline{\chi_1 \dots \chi}_k,p^{m-n}).
$$
If $\overline{\chi_1 \dots \chi}_k$ is a primitive character mod $p^j$ with $m-n < j \leq m$, then by the same reasoning as in \eqref{sumiszero}
$$
\sum_{\substack{y=1 \\ p\nmid y}}^{p^m} \overline{\chi_1 \dots \chi}_k (y) e_{p^m}(p^ny)=p^{m-j}\sum_{y=1 }^{p^{j}} \overline{\chi_1 \dots \chi}_k (y) e_{p^{j}}(p^{j-(m-n)}y) )=0
$$
and the result follows on observing that
$$ \overline{G(\chi,p^m)}= \overline{\chi}(-1) G(\overline{\chi},p^m). $$
\end{proof}

\section{Proof of Theorem \ref{maintheorem}}
We assume that $\chi_1,\ldots ,\chi_k$ are all primitive mod $p^m$ characters and $\chi_1 \cdots \chi_k$ is  a primitive mod $p^{m-n}$ character,  since otherwise from Theorem  \ref{intermsofgauss} and \eqref{gaussmod}, $J_{p^n}(\chi_1,\ldots ,\chi_k,p^m)=0$. In particular we have 
  \eqref{symmetricform2}.

Writing $R=R_{\lceil \frac{m}{2}\rceil +2}$ then  by \eqref{symmetricform2} and the evaluation of Gauss sums in Theorem \ref{gausssumthm} we have
\begin{align}
J_{p^n}(\chi_1,\ldots ,\chi_k,p^m) &  = 
\frac{ \prod_{i=1}^{k} G(\chi_i , p^m)}{G(\chi_1 \dots \chi_k,p^{m-n})}\notag \\
&= \frac{ \prod_{i=1}^{k} p^{m/2}\chi_i(-c_iR^{-1}) e_{p^m}(-c_iR^{-1})\delta_i}{p^{(m-n)/2}\chi_1 \dots \chi_k(-vR^{-1})e_{p^{m-n}}(-vR^{-1})\delta_s}\notag \\
& = p^{\frac{1}{2}(m(k-1)+n)} \frac{ \prod_{i=1}^{k} \chi_i(c_i)}{\chi_1 \dots \chi_k(v)}\delta_s^{-1} \prod_{i=1}^k \delta_i \label{deltaprod},
\end{align}
where
$$
\delta_i = \begin{cases}  \left( \frac{-2rc_i}{p} \right)^m\varepsilon_{p^m}, & \text{if  $p$ is odd, $p\neq 2$,}\\
\left(\frac{2}{c_i}\right)^m \omega^{c_i}, & \text{if $p=2$ and  $m\geq 5$,} \end{cases}
$$
and
$$
 \delta_s= \begin{cases} 
 \left( \frac{-2rv}{p} \right)^{m-n}\varepsilon_{p^{m-n}}, & \text{if $p$  is odd,}\\
\left(\frac{2}{v}\right)^{m-n} \omega^{v}, & \text{if $p=2$ and $m-n\geq 5$,} \end{cases}
$$
and the result is plain when $p$ is odd or $p=2$, $m-n\geq 5$.

 The remaining cases $p=2$, $m\geq 5$ and  $m-n=2,3,4$,   follows similarly using the adjustment to
$\delta_s $  observed at the end of the proof of  Theorem \ref{gausssumthm} .

\section{A more direct approach}

We should note that the Cochrane \& Zheng reduction technique \cite{cz}  can be applied to directly evaluate the Jacobi sums when $p$ is odd and $m\geq n+2$ instead of the Gauss sums. 
For example if $b=p^nb'$ with $p\nmid  b'$, then from \cite[ Theorem 3.1]{PPP}  we have
\begin{align*}
 J_b\left( \chi_1, \chi_2, p^m \right) &= \sum_{x=1}^{p^m} \chi_1(x) \chi_2(b-x)  = \sum_{x=1}^{p^m} \overline{\chi_1 \chi_2}(x)\chi_2(bx-1)
\\ &= p^{\frac{m+n}{2}} \overline{\chi_1 \chi_2}(x_0)\chi_2(bx_0-1) \left(\frac{ -2c_2rb'x_0}{p} \right)^{m-n} \varepsilon_{p^{m-n}}, 
\end{align*}
where  $x_0$ is a solution to the characteristic equation \be \label{characteristiceqn} c_1+c_2-c_1bx \equiv 0 \text{ mod }p^{\lfloor\frac{m+n}{2}\rfloor+1},\;\;\; p \nmid x(bx-1).\ee  
 If  \eqref{characteristiceqn}  has no solution mod  $p^{\lfloor\frac{m+n}{2}\rfloor}$  then $J_b(\chi_1,\chi_2,p^m)=0$. In particular we see that:
\begin{enumerate}
\item[i. ] If $p \nmid c_1$ and $p \mid c_2$, then $J_b(\chi_1,\chi_2,p^m)=0$.
\item[ii. ] If $p\nmid c_1 c_2(c_1+c_2)$ then
$$
J_b(\chi_1,\chi_2,p^m)= \chi_1\chi_2(b) \chi_1(c_1) \chi_2(c_2) \overline{\chi_1\chi_2}(c_1+c_2) p^\frac{m}{2} \delta_2.
$$
where $$\delta_2 =  \left( \frac{-2r}{p} \right)^m \left( \frac{c_1c_2(c_1+c_2)}{p} \right)^m \varepsilon_{p^m}.$$ 
\item[iii.] If $p\nmid c_1$ and $b=p^nb'$, $p\nmid b'$ with $n<m-1$ then $J_b(\chi_1,\chi_2,p^m)=0$ unless $p^n \mid\mid (c_1+c_2)$ in which case writing $w=(c_1+c_2)/p^n$,
$$   J_b(\chi_1,\chi_2,p^m)= \chi_1\chi_2(b') \frac{\chi_1(c_1) \chi_2(c_2)}{ \chi_1\chi_2(w)} p^\frac{m+n}{2}\left( \frac{-2r}{p} \right)^{m-n}\left( \frac{c_1c_2w}{p} \right)^{m-n} \varepsilon_{p^{m-n}} .         $$

\end{enumerate}

To see (ii) observe that if $p\mid b$,  then $J_b(\chi_1,\chi_2,p^m)=0$,  and if $p\nmid b$,  then we
can take 
$x_0 \equiv (c_1+c_2)c_1^{-1}b^{-1} \text{ mod } p^m$ (and hence $bx_0-1=c_2c_1^{-1}$).
Similarly for (iii) if $p^n \mid\mid (c_1+c_2)$ we can take 
$x_0 \equiv p^{-n}(c_1+c_2)c_1^{-1}(b')^{-1}$ mod $p^m$.

Of course  we can write the generalized sum in the form
\begin{align*}J_{p^n}(\chi_1, \dots, \chi_k) & = \sum_{x_3=1}^{p^m} \dots \sum_{x_k=1}^{p^m}\chi_3(x_3) \dots \chi_k(x_k) \sum_{\substack{x_1=1 \\ b:=p^n-x_3- \dots -x_k}}^{p^m} \chi_1(x_1) \chi_2(b-x_1)
\\& =\sum_{x_3=1}^{p^m} \dots \sum_{x_k=1}^{p^m}\chi_3(x_3) \dots \chi_k(x_k) J_b(\chi_1,\chi_2,p^m),
\end{align*}
Hence  assuming that at least one of the  $\chi_i$ is primitive mod $p^m$  (and reordering the characters as necessary) we see from (i) that $J_{p^n}(\chi_1,\ldots ,\chi_k,p^m)=0$  unless all the 
characters are  primitive mod $p^m$. Also when $k=2$, $\chi_1,\chi_2$ primitive, we see from (iii) that
$J_{p^n}(\chi_1,\chi_2,p^m)=0$ unless $\chi_1\chi_2$ is induced by a primitive mod $p^{m-n}$ 
character, in which case we recover the formula in Theorem \ref{maintheorem} on observing 
that $\left(\frac{c_1c_2}{p}\right)^n\varepsilon_{p^{m-n}}^2=\varepsilon_{p^m}^2$;
this is plain when $n$ is even, for $n$ odd observe that $\left(\frac{c_1c_2}{p}\right) =\left(\frac{(c_1+c_2)^2-(c_1-c_2)^2}{p}\right)=\left(\frac{-1}{p}\right)$. We show  that a simple induction  recovers the formula for all $k\geq 3$. 
We assume that all the $\chi_i$ are primitive mod $p^m$
and observe that when $k\geq 3$ we can further  assume (reordering as necessary) that $\chi_1\chi_2$ 
is also primitive mod $p^m$, since if $\chi_1\chi_3$, $\chi_2\chi_3$ are not primitive then $p\mid (c_1+c_3)$ and $p\mid (c_2+c_3)$ and $(c_1+c_2)\equiv -2c_3\not\equiv 0$ mod $p$ and $\chi_1\chi_2$ is primitive. Hence from (ii) we can write
\begin{align*}
J_{p^m}(\chi_1, \dots, \chi_k,p^m) &= \frac{\chi_1(c_1)\chi_2(c_2)}{\chi_1\chi_2(c_1+c_2)} p^\frac{m}{2} \delta_2 
\sum_{x_3=1}^{p^m} \dots \sum_{x_k=1}^{p^m}\chi_3(x_3) \dots \chi_k(x_k) \chi_1\chi_2(b)
\\ &= \chi_1(c_1)\chi_2(c_2)\overline{\chi_1\chi_2}(c_1+c_2) p^\frac{m}{2} \delta_2 J_{p^n}(\chi_1\chi_2, \chi_3, \dots, \chi_k,p^m).
\end{align*}
Assuming the result for $k-1$  characters we have $J_{p^n}(\chi_1\chi_2, \chi_3, \dots, \chi_k,p^m)=0$ unless $\chi_1\cdots \chi_k$ is induced by a primitive mod $p^{m-n}$ character
in which case
$$ J_{p^n}(\chi_1\chi_2, \chi_3, \dots, \chi_k,p^m)=  \chi_1\chi_2(c_1+c_2)  \prod_{i=3}^{k}\chi_i(c_i) \overline{\chi_1 \dots \chi_k}(v) \delta_3 p^{\frac{m(k-2)+n}{2}} $$
with $$\delta_3 =\left( \frac{-2r}{p} \right)^{m(k-2)+n} \left(\frac{v}{p}\right)^{m-n}\left( \frac{(c_1+c_2)c_3 \dots c_k}{p} \right)^m \varepsilon_{p^m}^{k-1}\varepsilon_{p^{m-n}}^{-1}.$$
Our formula for $k$ characters then  follows on  observing that  $\delta_2\delta_3 =\delta$.

\end{document}